\documentclass[12pt]{amsart}
\usepackage{amssymb,latexsym, amsmath, amscd, array, graphicx}

\swapnumbers \numberwithin{equation}{section}
\usepackage{mathtools}
\usepackage{mathrsfs}
\usepackage{enumerate}
\usepackage{mathabx}
\usepackage{setspace}
\usepackage{color}
\usepackage{comment}
\usepackage{tikz}
\usepackage{tikz-cd}

\theoremstyle{plain}

\newtheorem{thm}{Theorem}[section]
\newtheorem{theorem}[thm]{Theorem}
\newtheorem{lem}[thm]{Lemma}
\newtheorem{lemma}[thm]{Lemma}

\newtheorem{prop}[thm]{Proposition}
\newtheorem{cor}[thm]{Corollary}

\theoremstyle{definition}
\newtheorem{defin}[thm]{Definition}

\newtheorem{rem}[thm]{Remark}

\DeclareMathOperator{\cat}{{\mbox{\rm cat$_{\rm LS}$}}}

\def\Z{{\mathbb Z}}

\def\N{{\mathbb N}}

\def\1{\hbox{\rm\rlap {1}\hskip.03in{\rom I}}}
\def\Bbbone{{\rm1\mathchoice{\kern-0.25em}{\kern-0.25em}
{\kern-0.2em}{\kern-0.2em}I}}

\long\def\forget#1\forgotten{} %

\newcommand\ver[1]{\marginpar{\tiny Changed in Ver \VER}}

\date{\today}
\begin{document}

\title[On Iwase manifolds]{On Iwase's manifolds }

\author[A.~Dranishnikov]{Alexander  Dranishnikov} %
%

\address{Alexander N. Dranishnikov, Department of Mathematics, University
of Florida, 358 Little Hall, Gainesville, FL 32611-8105, USA}
\email{dranish@math.ufl.edu}

\subjclass[2000]
{Primary 55M30; 
Secondary 53C23,  
57N65  
}

\maketitle

\begin{abstract} 
In ~\cite{Iw2} Iwase has constructed two 16-dimensional manifolds $M_2$ and $M_3$ with LS-category 3 which are counter-examples to Ganea's conjecture: $\cat (M\times S^n)=\cat M+1$.
We show that the manifold $M_3$ is a counter-example to the logarithmic law for the LS-category of the square of a manifold: $\cat(M\times M)=2\cat M$. Also we construct a map of degree one
$$f:2(M_3\times S^2\times S^{14})\#-(M_2\times S^2\times S^{14})\to M_2\times M_3$$ which reduces  Rudyak's conjecture to the question whether $\cat(M_2\times M_3)\ge 5$ and
show that $\cat(M_2\times M_3)\ge 4$.
\end{abstract}

\section{Introduction}

The Lusternik-Schnirelmann category $\cat X$ is a celebrated numerical invariant of topological spaces $X$ which was introduced in the late 30s of the last century.
By  definition $\cat X$ is the minimal number $k$ such that $X$ can be covered by $k+1$ open sets $U_0,\dots, U_k$ such that each $U_i$ is contractible in $X$.
This invariant is of special importance when $X$ is a closed manifold, since it brings a lower bound on the number of critical points of  smooth functions on $X$~\cite{CLOT} Section 1.3.

It turns out that $\cat$ behaves differently with respect to two basic operations on manifolds, the connected sum and the cartesian product. In the case of connected sum there is a natural formula~\cite{DS1},\cite{DS2}: $$(\ast)\ \ \ \cat(M\# N)=\max\{\cat M,\cat N\}.$$ In the case of the product the LS-category behavior is weird.
There was the upper bound formula $$\cat (X\times Y)\le\cat X+\cat Y$$ since the late 30s~\cite{Ba},~\cite{F}. 
Since then there was a natural question for which spaces $X$ and $Y$ this upper bound is attained.
In particular, there was a  longstanding  conjecture
of Ganea that $$\cat(X\times S^n)=\cat X+1$$ for all $X$.
The Ganea Conjecture was verified for many classes of spaces $X$, yet it turned out to be false.
At the end of the last century
Noiro Iwase had constructed counterexamples to Ganea's conjecture, first when $X$ is a finite complex~\cite{Iw} and then  when $X$ is a closed manifold~\cite{Iw1},\cite{Iw2}.
He constructed two 16-dimensional manifolds denoted by $M_2$ and $M_3$  satisfying $$\cat(M_i\times S^n)=\cat M_i=3$$ for sufficiently large $n$, $i=2,3$.
Also Iwase proved that $$\cat(M\times S^n)=\cat M$$ for all $n$ either with $M=M_3$ or with $M=M_3\times S^1$.
The reason that Iwase manifolds have indexes 2 and 3 is that their constructions are related to the 2-primary and the 3-primary components of the group $\pi_{13}(S^3)=\mathbb Z_2\oplus\mathbb Z_4\oplus\mathbb Z_3$
respectively.

In this paper we exhibit a relation of Iwase's examples to the following two intriguing problems in the LS-category. The first is a problem from~\cite{Ru1} which is also known as the Rudyak Conjecture: 

{\em For a degree one map $f:M\to N$  between closed manifolds, $\cat M\ge\cat N$.}

The second is the question (which I also attribute to Yu. Rudyak): 

{\em Does the equality $\cat(M\times M)=2\cat M$ hold true for closed manifolds $M$?}

There were partial results on both problems. Thus, the Rudyak conjecture was proven in some special cases~\cite{Ru1},\cite{Ru2},\cite{DSc}. A finite 2-dimensional complex
that does not satisfy the equality $\cat(X\times X)=2\cat X$ were constructed in~\cite{H},~\cite{St2}.

In this paper we connect both problems to Iwase's examples. In particular, we prove the following
\begin{thm}
$$\cat(M_3\times M_3)\le 5<2\cat M_3.$$
\end{thm}
We connect the Rudyak Conjecture with computation of $\cat(M_2\times M_3)$. Namely, we show that if
$\cat(M_2\times M_3)\ge 5$, then there is  a counterexample to the Rudyak Conjecture. 
In this paper we managed to  prove only the inequality $\cat(M_2\times M_3)\ge 4$.

\section{Preliminaries}

\subsection{Ganea-Schwarz approach to the LS-category}
 An element of an iterated join $X_0*X_1*\cdots*X_n$ of topological spaces is a formal linear combination $t_0x_0+\cdots +t_nx_n$ of points $x_i\in X_i$ with $\sum t_i=1$, $t_i\ge 0$, in which all terms of the form $0x_i$ are dropped. Given fibrations $f_i: X_i\to Y$ for $i=0, ..., n$, the fiberwise join of spaces $X_0, ..., X_n$ is defined to be the space
\[
    X_0*_Y\cdots *_YX_n=\{\ t_0x_0+\cdots +t_nx_n\in X_0*\cdots *X_n\ |\ f_0(x_0)=\cdots =f_n(x_n)\ \}.
\]
The fiberwise join of fibrations $f_0, ..., f_n$ is the fibration 
\[
    f_0*_Y*\cdots *_Yf_n: X_0*_YX_1*_Y\cdots *_YX_n \longrightarrow Y
\]
defined by taking a point $t_0x_0+\cdots +t_nx_n$ to $f_i(x_i)$ for any $i$. As the name `fiberwise join' suggests, the fiber of the fiberwise join of fibrations is given by the join of fibers of fibrations. 

When $X_i=X$ and $f_i=f:X\to Y$ for all $i$  the fiberwise join of spaces is denoted by $*^{n+1}_YX$ and the fiberwise join of fibrations is denoted by $*_Y^{n+1}f$. 
For a path connected topological space $X$, we turn an inclusion of a point $*\to X$ into a fibration $p_0^X:G_0X\to X$. The $n$-th Ganea space of $X$ is defined to be the space $G_nX=*_X^{n+1}G_0X$, while the $n$-th Ganea fibration $p_n^X:G_nX\to X$ is the fiberwise join of fibrations $p_0^X:G_nX\to X$.
Thus, the fiber $F_nX$ of $p_n^X$ is the iterated join product $F_nX=\ast^{n+1}\Omega X$ of the loop space of $X$.

 The following theorem is called Ganea's characterization of the LS-category [Ga]. It was proven first in~\cite{Sch} in a greater generality.
\begin{theorem}  
 Let $X$ be a connected  CW-complex. Then $\cat(X)\le n$ if and only if the fibration $p_n^X:G_nX\to X$ admits a section. 
\end{theorem}
There is a chain of inclusions $G_0X\subset G_1X\subset G_2X\subset\dots$ such that each $p^X_n$ is a restriction of $p^X_{n+1}$ to $G_nX$.

We recall that a map $f:X\to Y$ is called an \emph{$n$-equivalence} if it induces isomorphisms of homotopy groups $f_*:\pi_i(X)\to\pi_i(Y)$ for $i<n$ and an epimorphism for $i=n$.

\begin{prop}[\cite{DS2}]\label{old}
 Let $f: X\to Y$ be an $s$-equivalence of pointed $r$-connected CW-complexes with $r\ge 0$. Then the induced map $f_k:F_kX\to F_kY$ of the fibers of the $k$-th Ganea spaces is a $(k(r+1)+s-1)$-equivalence. 
\end{prop}

The following theorem was proven by D. Stanley~\cite[Theorem 3.5]{St1}.
\begin{theorem}\label{Stanley}
 Let $X$ be a connected CW-complex with $\cat X=k>0$. Then $\cat X^{(s)}\le k$ for any $s$, where  $X^{(s)}$ is the $s$-skeleton of $X$.
\end{theorem}
\begin{proof}
Since the inclusion $j:X^{(s)}\to X$ is an $s$-equivalence, by Proposition~\ref{old} the induced map $j_k:F_kX^{(s)}\to F_kX$
of the fibers of the $k$-th Ganea fibrations  is a $(k+s-1)$-equivalence.
Since   $s\le k+s-1$, this implies that a section of $p_k^X$ defines a section of $p_k^{X^{(s)}}$.
\end{proof}

\subsection{Berstein-Hilton invariant}
The Ganea spaces $G_nX$ admit a reasonable homotopy theoretic description~\cite{CLOT}. We mention that $G_1X$ is homotopy equivalent to $\Sigma\Omega X$ where $\Omega$ denote the loop space and $\Sigma$ the reduced suspension. Moreover, $p^X_1$ is homotopic to the composition $e\circ h$ where $h:G_1X\to\Sigma\Omega X$ is a homotopy equivalence and $e:\Sigma\Omega X\to X$ is the evaluation map, $e([\phi,t])=\phi(t)$.

There is a natural inclusion $i_Y:Y\to\Omega\Sigma Y$ which takes $y\in Y$ to the meridian loop $\psi_y$ through $y$. Note that $i_{\Omega X}$ defines a section of
the map $\Omega(e):\Omega\Sigma\Omega X\to \Omega X$. This proves the following
\begin{prop}
The loop fibration $\Omega p^X_n:\Omega G_nX\to\Omega X$ admits a canonical section $s:\Omega(X)\to\Omega G_nX$.
\end{prop}
\begin{cor}\label{split}
The induced map $(p^X_n)_*:\pi_k(G_nX)\to\pi_k(X)$ has a natural splitting for all $n\ge 2$.
\end{cor}
We denote by $r_*:\pi_k(G_nX)\to \pi_k(F_nX)$ the projection defined by the above splitting.

\begin{defin}\cite{BH}
Let $\alpha\in\pi_k(X)$. We define the Berstein-Hilton invariant of $\alpha$ as the family
$$
H_k(\alpha)=\{r_*\sigma_*(\alpha)\in\pi_k(F_nX)\mid \sigma:X\to G_kX \ \text{is a section}\}.
$$
We use notation $H^\sigma_k(\alpha)\in H_k(\alpha)$ for a representative of the homotopy class $[r_*\sigma_*(\alpha)]\in \pi_k(F_nX)$.
\end{defin}
Since $e:\Sigma\Omega S^n\to S^n$ has a unique homotopy section, $H_1([f])$ consists of one element for any $f:S^k\to S^n$.

By the James decomposition formula $\Sigma\Omega S^2$ is homotopy equivalent to the wedge $\bigvee_{i=1}^\infty S^{i+1}$ and for any $\alpha\in\pi_k(S^2)$
the Berstein-Hilton invariant  $H_1(\alpha)$ is the collection of the j-th James-Hopf invariants (see \cite{CLOT}, Section 6.2)
$h_j(\alpha)\in\pi_k(S^{j+1})$, $j\ge 2$. Thus, $H_1(\eta)=1\in\pi_3(S^3)=\mathbb Z$ for the Hopf bundle $\eta:S^3\to S^2$.

The Berstein-Hilton invariant can be helpful in determining whether $$\cat(X\cup_\alpha D^{n+1})\le\cat X$$ where $\alpha:\partial D^{n+1}\to Y$ is the attaching map~\cite{BH},\cite{Iw1},\cite{CLOT}:
\begin{thm}\label{BH}
Let $\cat X=k$. If $ H_k(\alpha)$ contains 0,  then $$\cat(X\cup_\alpha D^{n+1})\le k.$$
The converse holds true whenever $\dim X\le k+n-2$.
\end{thm}

\subsection{Fibration-Cofibration Lifting problem}
Let
$$
\begin{CD}
FY@>i>>\bar Y @>p>>  Y\\
@. @As AA @ AgAA\\
S^n @>f>> X @>j>\subset> C(f)\\
\end{CD}
$$
be a  homotopy commutative diagram where the top row is a fibration $p:\bar Y\to Y$ with the fiber $FY$ and the bottom row is a cofibration sequence. 
The  Fibration-Cofibration Lifting problem is to find a homotopy lift $\bar s:C(f)\to\bar Y$ of $g$ extending $s$.

\begin{prop}\label{extract}
Suppose that the inclusion homomorphism $i_*:\pi_n(FY)\to\pi_n(\bar Y)$ is injective.
Then there is unique up to homotopy map $\phi:S^n\to FY$ that makes the diagram
$$
\begin{CD}
FY@>i>>\bar Y @>p>>  Y\\
@A\phi AA. @As AA @ AgAA\\
S^n @>f>> X @>j>\subset> C(f)\\
\end{CD}
$$
 homotopy commutative.

 The Fibration-Cofibration Lifting problem has a solution if and only if $\phi$ is null-homotopic.
\end{prop}
\begin{proof}
A canonical homotopy of $jf$ to a constant map defines a homotopy $H:S^n\times I\to Y$ of  $gjf$ to a constant map. There is a  lift $\bar H:S^n\times I\to\bar Y$ with $\bar H|_{S^n\times 0}$ homotopic to $sf$.
Then $\phi=\bar H|_{S^n\times 1}$. The injectivity condition implies the uniquenes of $\phi$.

Clearly, a lift $\bar s:C(f)\to\bar Y$ defines the lift $\bar H$ such that $\phi=\bar H|_{S^n\times 1}$ is a constant map to $FY$.

Now assume that $\phi=\bar H|_{S^n\times 1}$ is null-homotopic. Then the homotopy $\bar H$, and a contraction of $\bar H|_{S^n\times 1}$ to a point in $FY$ define $\bar s$.

\end{proof}

\begin{rem}\label{rema}
When $\cat X\le k$ we fix a section $\sigma:X\to G_kX$ of $p_k^X$ and consider
 the following Fibration-Cofubration Lifting problem. 
$$
\begin{CD}
F_kY@>i>>G_kY @>p>>  Y\\
@A\phi AA. @As AA @ A=AA\\
S^n @>f>> X @>j>\subset> C(f)\\
\end{CD}
$$
where $s=G_k(j)\circ\sigma$ and $j:X\to Y=X\cup_\alpha D^{n+1}$ is the inclusion.
Since $p_k$ is a split surjection for homotopy groups, the  condition of Proposition~\ref{extract} is satisfied.
Moreover, $\phi$ is homotopic to $F_k(j)\circ H_k(\alpha)$.  This explains the first part of Theorem~\ref{BH}.
By Proposition~\ref{old}, $F_k(j):F_kX\to F_kY$ is a $(k+n-1)$-equivalence. This explains the second  part.
\end{rem}

\subsection{The LS-category of spherical bundles} 
Let $q:M\to S^{t+1}$ be a locally trivial bundle with fiber $S^r$. Let $g:(D^{t+1},\partial D^{t+1})\to (S^{t+1},x_0)$ be the quotient map. Then $M$ is the image under the pull-back map $\bar g:D^{t+1}\times S^r\to M$.
Let $\Psi:S^t\times S^r\to S^r$ be the restriction of $\bar g$ to $\partial D^{t+1}\times S^r$.
Then $M=S^r\cup_\Psi  (D^{t+1}\times S^r$). We consider the standard CW complex structures on $D^{t+1}=S^t\cup e^{t+1}$ and $S^r=\ast\cup e^r$. Then
$D^{t+1}\times S^r=(S^t\times S^r)\cup (e^{t+1}\times\ast)\cup (e^{t+1}\times e^r)$.  Thus,  we obtain $$M=S^r\cup_\alpha e^{t+1}\cup_\psi e^{r+t+1}$$ with $\alpha=\Psi|_{S^t\times \ast}$ where the spere $S^r$ is identified with $q^{-1}(x_0)$.

If $q':M'\to S^{t'+1}$ is the pull-back of $q$ with respect to the suspension of a map $f:S^{t'}\to S^t$, then $\Psi'$ factors through $\Psi$ and $\alpha'=\alpha\circ f$.

Clearly, $\cat M\le 3$. The category of $M$ and in some cases the category of the product $M\times S^n$ can be computed in terms of the attaching map $\alpha$ in view of the following.
\begin{thm}[\cite{Iw2}]\label{cat}
Let $t>r>1$ and $H_1(\alpha)\ne 0$. Then
\begin{itemize}
\item $\cat M=3 $ if and only if $\Sigma^rH_1(\alpha)\ne 0$.
\item $\Sigma^{n+r}H_1(\alpha)=0$ implies $\cat(M\times S^n)=3$.
\item $\Sigma^{n+r+1}h_2(\alpha)\ne 0$ implies $\cat(M\times S^n)=4$.
\end{itemize}
\end{thm}

\subsection{Homotopy groups of spheres} We follow the notations from Toda's book~\cite{T}. 
The Hopf bundles $\eta:S^3\to S^2$, $\nu:S^7\to S^4$, and $\sigma:S^{15}\to S^8$ produce by suspensions the elements $\eta_n\in\pi_{n+1}(S^n)$, $\nu_n\in\pi_{n+3}(S^n)$, and
$\sigma_n\in\pi_{n+7}(S^n)$. We use the notation $\eta_n^2$ for the composition $\eta_n\circ\eta_{n+1}:S^{n+2}\to S^n$ as well as for the generator of $\pi_{n+2}(S^n)=\mathbb Z_2$.
The generator $\epsilon_3$ of the $\mathbb Z_2$ summand of $\pi_{11}(S^3)$ and its suspensions produce generators $\epsilon_n\in\pi_{n+8}(S^n)$. 
\begin{prop}\label{toda}
Let $\phi=\eta_3^2\circ\epsilon_5\in\pi_{13}(S^2)$. Then $\Sigma^5\phi\ne 0$ and $\Sigma^6\phi=0$.
\end{prop}
\begin{proof} For $n\ge 2$, $\Sigma^n\phi=\eta^2_{n+3}\circ\epsilon_{n+5}=4(\nu_{n+3}\circ\sigma_{n+6})$ by (7.10) of~\cite{T}.
By  Theorem 7.3 (2)~\cite{T}, $\nu_{n+3}\circ\sigma_{n+6}$ generates the subgroup $\mathbb Z_8\subset\pi_{n+13}(S^{n+3})$ for $n=2,3,4,5$. Hence, $\Sigma^5\phi\ne 0$.
By  (7.20) in~\cite{T} $\Sigma^6\phi=4(\nu_9\circ\sigma_{12})=0$.
\end{proof}

We recall some facts about primary $p$-components  $\pi_i(S^n;p)$ of  homotopy groups $\pi_i(S^n)$ for odd prime $p$. Namely,
for $i\in\{1,2,\dots,p-1\}$,  $m\ge 1$,
$$\pi_{2m+1+2i(p-1)-1}(S^{2m+1};p)=\mathbb Z_p$$
with the generators  $\alpha_i(2m+1)$ satisfying the condition $\Sigma^2\alpha_i(2m-1)=\alpha_i(2m+1)$.  Using suspension we define $\alpha_i(n)$ for even $n$ as well.
Then the group $\pi_{2i(p-1)+1}(S^3;p)\cong\mathbb Z_p$ (see ~\cite{T} Proposition 13.6.)
for $2\le i\le p$  is generated by $\alpha_1(p)\circ\alpha_{i-1}(2p)$. 
There is Serre isomorphism $$\pi_i(S^{2m};p)\cong \pi_{i-1}(S^{2m-1};p)\oplus\pi_i(S^{4m-1};p)$$
such that the suspension $\Sigma:\pi_{i-1}(S^{2m-1};p)\to\pi_i(S^{2m};p)$  defines the embedding of the first summand.

\begin{prop}\label{toda-p}
Let $\psi=\alpha_1(3)\circ\alpha_2(6)$. Then $\Sigma^3\psi\ne 0$ and $\Sigma^4\psi=0$.
\end{prop}
\begin{proof}
By $(13.6)'$  in~\cite{T} the group $\pi_{13}(S^3;3)=\mathbb Z_3$ is generated by $\psi$.
By Theorem 13.9~\cite{T}, $\pi_{14}(S^3;3)=\mathbb Z_3$, $\pi_{16}(S^5;3)=\mathbb Z_9$, and $\pi_{15}(S^5;3)=\mathbb Z_9$.
The exact sequence (13.2) from~\cite{T} in view of (13.6) produces the exact sequence
$$
0\to\pi_{14}(S^3;3)\stackrel{\Sigma^2}\to\pi_{16}(S^5;3)\to\mathbb Z_3\to\pi_{13}(S^3;3) \stackrel{\Sigma^2}\to\pi_{15}(S^5;3)\to\mathbb Z_3
$$
which implies that $\pi_{13}(S^3;3) \stackrel{\Sigma^2}\to\pi_{15}(S^5;3)$ is injective.
Therefore, $\Sigma^2\psi$ generates a subgroup $\mathbb Z_3\subset\mathbb Z_9=\pi_{15}(S^5;3)$. In particular,   $\Sigma^2\psi\ne 0$.
By the Serre isomorphism, the suspension homomorphism $$\pi_{15}(S^5;3)\to \pi_{16}(S^6;3)$$  is a monomorphism. Hence, $\Sigma^3\psi\ne 0$.

The exact sequence (13.2) from~\cite{T} implies that the following sequence
$$
\mathbb Z_3\to\pi_{15}(S^5;3)\stackrel{\Sigma^2}\to\pi_{17}(S^7;3)\to 0
$$
is exact. Since   $\pi_{17}(S^7;3)=\mathbb Z_3$ (Theorem 3.19~\cite{T}), this implies that $\Sigma^4\psi=0$.
\end{proof}

\section{Iwase's Examples}

\subsection{Manifold $M_2$} The $S^1$-action on $S^7$ defines a factorization of the Hopf bundle  $\nu_4:S^7\to S^4$  through the $S^2$-bundle $h:\mathbb C P^3\to S^4$.
Iwase defined $M_2$ as the total space of the $S^2$-bundle $q_2:M_2\to S^{14}$ induced from $h$ by means of the suspension map $f_2=\Sigma f_2'$ where $f_2'$ represents $\eta^2_3\circ\epsilon_5\in\pi_{13}(S^3)$. Then the gluing map $\Psi:S^{13}\times S^2\to S^2$ for $M_2$ is the composition $\Psi_0\circ f_2'$ where $\Psi_0:S^3\times S^2\to S^2$ is the gluing map for $h$. Then the attaching map for the 14-cell in $M_2$ is the composition $\alpha=\alpha_0\circ f'_2$ where $\alpha_0$ is the attaching map of the 4-cell in $\mathbb C P^2$. Thus, $\alpha$ represents $\eta\circ\eta^2_3\circ\epsilon_5$:
$$
\begin{CD}
S^{13} @>\epsilon_5>>S^5 @>\eta_4>> S^4 @>\eta_3>> S^3 @>\eta>> S^2.
\end{CD}
$$
The following is a minor refinement of Iwase's theorem~\cite{Iw2}.
\begin{prop}\label{M2}
The manifold $M_2$ has the following properties:
\begin{enumerate}
\item $\cat (M_2\times S^n)=3$ for $n\ge 4$;

\item $\cat M_2=3$;

\item $\cat (M_2\times S^1)=\cat(M_2\times S^2)=4$;

\item There is a map $f:S^{14}\times S^2\to M_2$ of degree 2.
\end{enumerate}
\end{prop}
\begin{proof}
To prove (1)-(3) we show that $H_1(\alpha)=h_2(\alpha)=\phi$. Then by Proposition~\ref{toda}, $\Sigma^2H_1(\alpha)\ne 0$, $\Sigma^6H_1(\alpha)=0$, and $\Sigma^5h_2(\alpha)\ne 0$. Theorem~\ref{cat} implies that $\cat M_2=3$, $\cat(M_2\times S^n)=3$ for $n\ge 4$ and $\cat(M_2\times S^1)=\cat(M\times S^2)=4$.

Let $i_X\to \Omega\Sigma X$ denote the natural inclusion.
If $\beta=\eta\circ \gamma$ were  $\gamma=\Sigma\gamma'$ is a suspension, $\gamma':S^{k}\to S^r$, then the commutativity of diagram

\[
\begin{tikzcd}
\Sigma\Omega\Sigma S^k \arrow[r, "\bar\gamma_1"]  \arrow[d]
& \Sigma\Omega\Sigma S^r\arrow[r, "\bar\eta_1"]  \arrow[d]  &\Sigma\Omega S^2 \arrow[d, "p_1^{S^2}"]\\
\Sigma S^k \arrow[r,"\gamma"]  \arrow[u,   "\Sigma  i_{S^k}", dashed, shift left=1.5ex]
&   \Sigma S^r  \arrow[u,  "\Sigma i_{S^r}", dashed, shift left=1.5ex]\arrow[r, "\eta"] & S^2\arrow[u,  dashed, shift left=1.5ex]
\end{tikzcd}
\]
implies that $H_1(\beta)=H_1(\eta)\circ\gamma$ and $h_j(\beta)=h_j(\eta)\circ\gamma$. If
 $\eta:S^3\to S^2$, then $h_j(\eta)=0$ for $j\ge 3$, $h_2(\eta)=H_1(\eta)$, and $h_2(\beta)=H_1(\beta)$.

Note that $\eta_3^2\circ\epsilon_5$ is a suspension since by definition $\eta_n^2$ and $\epsilon _m$ are suspensions for $n>2$ and $m>3$,
and for the Hopf map, $H_1(\eta)=1$.
Then, $h_2(\alpha)=H_1(\alpha)=H_1(\eta)\circ (\eta_3^2\circ\epsilon_5)=\eta_3^2\circ\epsilon_5=\phi\in\pi_{13}(S^3)$ for the map $\phi$ defined in Proposition~\ref{toda}.

Proof of (4). Let $2:S^{14}\to S^{14}$ be a map of degree 2. It induces the map of the pull-back manifold $f:M'\to M_2$ of degree 2. Note that $M'$ is the pull-back of $\mathbb C P^3$
with respect to the map $f_2\circ q$ which defines zero element of $\pi_{14}(S^4)$, since $\Sigma(\eta_3^2\circ\epsilon_5)\circ 2=\eta_4^2\circ\epsilon_6\circ 2=(2\eta_4^2)\circ\epsilon_6=0$ in view of the equality $2\eta_n=0$ for $n>2$. Therefore, $M'$ is homeomorphic to $S^{14}\times S^2$.
\end{proof}

\

\subsection{Manifold $M_3$} The manifold $M_3$ is defined as the total space of the $S^2$-bundle $q_3:M_3\to S^{14}$ induced from $h:\mathbb C P^3\to S^4$ 
by means of the suspension map $f_3=\Sigma f_3'$ where $f_3':S^{13}\to S^3$ is a map representing $\alpha_1(3)\circ\alpha_2(6)$.
 Then as in the construction of $M_2$ the attaching map $\alpha=\alpha_0\circ f'_3$ where $\alpha_0$ is the attaching map of the 4-cell in $\mathbb C P^2$. Thus, $\alpha$ represents $\eta_2\circ\alpha_1(3)\circ\alpha_2(6)$:
$$
\begin{CD}
S^{13} @>\Sigma\alpha_2(5)>>S^6  @>\alpha_1(3)>> S^3 @>\eta_2>> S^2.
\end{CD}
$$
\begin{prop}\label{M3}
The manifold $M_3$ has the following properties:

(1) $\cat (M_3\times S^n)=3$ for $n\ge 2$;

(2) $\cat M_3=3$;

(3) There is a map $f:S^{14}\times S^2\to M_3$ of degree 3.
\end{prop}
\begin{proof}
Properties (1) and (2) were proven in~\cite{Iw1} by application of Theorem~\ref{cat}. Namely, it was shown that $H_1(\alpha)=h_2(\alpha)=\psi$ and then  Proposition~\ref{toda-p}
was applied.

The argument for this is similar to the argument in the proof of Proposition~\ref{M2} with the difference is that $\alpha_1(3):S^6\to S^3$ is not a suspension. It turns out that $\alpha_1(3)$ is a co-H map
and this is sufficient to get the equality $H_1(\alpha)=H_1(\eta)\circ(\alpha_1(3)\circ\Sigma\alpha_2(5))$.

Proof of (3) is similar to the proof of (4) in Proposition~\ref{M2} and it is based on the fact that $\alpha_1(3)$ has the order 3.
\end{proof}

\section{Category of  the product}

It is known~\cite{Iw},\cite{SS} that in Ganea's definition of the category of the product of two complexes $X\times Y$ instead of 
the Ganea fibrations $p_k:G_k(X\times Y)\to X\times Y$
one can take a map with the smaller domain
$$\hat G_k(X\times Y)=\bigcup_{i+j=k}G_iX\times G_jY\subset G_kX\times G_kY.$$ There is the natural projection $\hat p_k:\hat G_k(X\times Y)\to X\times Y$
with fibers $$\hat F_k(X\times Y)=\bigcup_{i+j=k}F_iX\times F_jY\subset F_kX\times F_kY.$$
It is known that $\cat (X\times Y)\le k$ if and only if $\hat p_k$ admits a homotopy section~\cite{Iw}.
Though it is not important for our main result, we give a scketch of proof that in the case of CW complexes $X$ and $Y$ a homotopy section of $\hat p_k$ can be replaced by a section.
\begin{prop}\label{fibration}
For  CW complexes $X$ and $Y$ the map $$\hat p_k:\hat G_k(X\times Y)\to X\times Y$$ is a Serre fibration
\end{prop}
We recall that a map is claled a Serre fibration if it satisfies the homotopy lifting property for CW complexes.
This is equivalent to have the homotopy lifting property for $n$-cubes for all $n$.

The proof of Proposition~\ref{fibration} is based on the following two facts.

We say that a map $p:E\to B$  satisfies the {\em Homotopy Lifting Property for a pair} $(X,A)$ if for any
homotopy $H:X\times I\to B$ with a lift $H':A\times I\to E$ of the restriction $H|_{A\times I}$
and a lift $H_0$ of $H|_{X\times0}$ which agrees with $H'$, there is a lift $\bar H:X\times I\to E$
of $H$ which agrees with $H_0$ and $H'$.
The following is well-known~\cite{Ha}:
\begin{thm}\label{lift}
Any Serre fibration $p:E\to B$  satisfies the Homotopy Lifting Property for CW complex pairs $(X,A)$.
\end{thm}
We use the abbreviation ANE for absolute neighborhood extensors for the class of finite dimensional spaces. Such spaces can be characterized as those which are locally $n$-connected for all $n$.
\begin{lem}[Pasting Lemma for Fibrations]
Let $p:E\to B$ be a map between ANE spaces with closed subsets $E_1,E_2\subset E$ such that all spaces $ E_1,E_2,E_1\cap E_2$ are ANE. Suppose that the restrictions $p_1=p|_{E_1}:E_1\to B$, $p_2=p|_{E_2}:E_2\to B$, and $p_0=p|_{E_1\cap E_2}:E_1\cap E_2\to B$,
are Serre fibrations. Then  the restriction $\hat p=p|_{E_1\cup E_2}:E_1\cup E_2\to B$ is a Serre fibration.
\end{lem}
\begin{proof}
Let $H:Z\times I\to B$ be a homotopy of a finite CW complex $Z$ with a fixed lift $\bar h:Z\to E_1\cup E_2$ of $H|_{Z\times\{0\}}$.
Let $Z_1=\bar h^{-1}(E_1)$, $Z_2=\bar h^{-1}(E_2)$, and  $Z_0=\bar h^{-1}(E_1\cap E_2)$.
If it happens to be that $Z_0$ is a CW complex we apply homotopy lifting property of $p_0$ to $H|_{Z_o\times I}$. Then we apply Theorem~\ref{lift} twice, first for to $p_1$ and then for
$\hat p$ and we are done. 

For general $Z_0$ we apply the standard trick scatched below. We use  the ANE property to find a CW complex neighborhood $A$ of $Z_0$ and  a map  $\hat h:Z\to E_1\cup E_2$  with $\hat h(A)\subset E_1\cap E_0$,
and with a small homotopy $\hat h_t:Z\to E_1\cup E_2$ joining $\bar h$ and $\hat h$ and stationary on $Z_0$. Using local contractibility property of $B$ we may achive that
the homotopy $\hat h_t$ is fiberwise. Thus, we obtain a lift of $H$  if our homotopy $H$ is stationary on a small neihborhood of $Z\times\{0\}$. Generally, we change the fiberwise homotopy $\hat h_t$ by pushing it along the 
paths of $H$ with such control  that $\hat h_1(A)\subset E_1\cap E_2$.
\end{proof}

Now the proof of Proposition~\ref{fibration} can be given by induction using the Pasting Lemma and the fact that
for a CW-complex $X$ all the spaces $G_kX$ are ANE.

There are two important facts  about $\hat p_k$~\cite{Iw}: There is a  lift $\mu:G_k(X\times Y)\to\hat G_k(X\times Y)$ of $p_k$ with respect to $\hat p_k$
and there is the inequality $\cat\hat G_k(X\times Y)\le k$. The latter can be proven by the cone length estimate (see~\cite{SS}).
The inequality $\cat\hat G_k(X\times Y)\le k$ implies that there is a lift 
$$\lambda=\lambda_{k,X,Y}:\hat G_k(X\times Y)\to G_k(X\times Y)$$ of $\hat p_k$ with respect to $p_k^{X\times Y}:G_k(X\times Y)\to X\times Y$.

Since $p_k$ is a split surjection for homotopy groups, the lift $\mu$ ensures that $\hat p_k$ is also a split surjection for homotopy groups.
Hence the inclusion homomorphism $\pi_n(\hat F_k(X\times Y))\to\pi_n(\hat G_k(X\times Y))$ is injective for all $n$.

\

Since the inclusions $F_rZ\to F_{r+1}Z$ are null-homotopic for all $Z$ and $k$, there are natural maps $$\eta_{i,j}:F_iX\ast F_jY\to (F_{i+1}X\times F_jY)\cup (F_iX\times F_{j+1}Y)\subset \hat F_{i+j+1}.$$

Let $Q=S^2\cup_\alpha D^{14}$ with $\alpha:S^{13}\to S^2$ from the construction of  $M_3$.
\begin{thm}\label{Iw}
Iwase's manifold $M_3$ satisfies the inequality
$$
\cat (M_3\times M_3)\le 5<2\cat M_3.
$$
\end{thm}
\begin{proof}
We recall that $M_3=S^2\cup_\alpha e^{14}\cup_\psi e^{16}=Q\cup e^{16}$. Consider the product CW-complex structure on $M_3\times M_3$.
Define a equence of CW subspaces
$$
X_1\subset X_2\subset X_3\subset X_4=M_3\times M_3
$$
as follows

$X_1=(Q\times S^2)\cup (S^2\times Q)$,

$X_2=X_1\cup (M_3\times\ast)\cup(\ast\times M_3)=X_1\cup (e^{16}\times\ast)\cup (\ast\times e^{16})$,

$X_3=(M_3\times S^2)\cup(S^2\times M_3)\cup (Q\times Q)=X_2\cup(e^{16}\times e^2)\cup (e^2\times e^{16})\cup (e^{14}\times e^{14})$,

$X_4=(M_3\times Q)\cup (Q\times M_3)=X_2\cup (e^{16}\times e^{14})\cup(e^{14}\times e^{16})$, and

$X_5=M_3\times M_3=X_3\cup (e^{16}\times e^{16})$.

It suffices to prove the inequality $\cat X_3\le 3$. Then since $X_4$ is obtained from $X_3$  by attaching cells to $X_3$, we get $\cat X_4\le 4$. Finally,  $\cat X_4\le 5$.
\end{proof}

\begin{prop}\label{IwQ3}
 $\cat X_3\le 3$.
\end{prop}
\begin{proof} 
Note that $X_3$ is obtained by attaching three cells to $X_2$, two cells of dimension 18 and one of dimension 28.
The attaching map $\alpha$ defines the attaching map $\bar\alpha$ in $$Q\times Q=X_1\cup_{\bar\alpha}e^{28}.$$
We denote by $\tilde\alpha$ the attaching map of the 28-cell to $X_2$. Thus $\tilde\alpha$ is the composition $\bar\alpha$ and the inclusion $X_1\subset X_2$.
Let $$\bar\psi:S^{17}\to (M_3\times\ast)\cup(Q\times S^2)$$ denote the attaching map of the top cell in $M_3\times S^2$. 
Note that $$X_2=((M_3\times\ast)\cup(Q\times S^2))\cup((S^2\times Q)\cup(\ast\times M_3)).$$
Then the attaching maps of the 18-cells in $X_3$
can be presented as $\psi_-=i_-\circ\bar\psi$ and $\psi_+=i_+\circ\bar\psi$ where $i_\pm$ are two symmetric inclusions of $(M_3\times\ast)\cup(Q\times S^2)$ into $X_2$.

We consider the embeddings $G_3M_3\times\ast\to G_3M_3\times G_0M_3$ and $G_2Q\times G_1S^2\to G_2M_3\times G_1M_3$ generated by the inclusions $\ast,S^2,Q\subset M_3$.
Then $$\bar X_2=(G_3M_3\times\ast)\cup (G_2Q\times G_1S^2)\cup (G_1S^2\times G_2Q)\cup(\ast\times M_3)$$ is embedded in $\hat G_3(M_3\times M_3)$. 
Let $\bar i$ denotes the embedding.
There is a natural projection $p'$ of $\bar X_2$ onto
$X_2=(M_3\times\ast)\cup (Q\times S^2)\cup (S^2\times Q)\cup(\ast\times M_3)$ that makes 
a commutative diagram
$$
\begin{CD}
\bar X_2 @>\bar i>> \hat G_3(M_3\times M_3)\\
@Vp'VV @V\hat p_3VV\\
X_2 @>\subset>> M_3\times M_3.\\
\end{CD}
$$
We define a section $\hat s:X_2\to\bar X$ of $p'$ as follows. 

Let $\psi:S^{15}\to Q$ be the attaching map in the construction of $M_3=S^2\cup_\alpha e^{14}\cup_\psi e^{16}$.
It was proven in~\cite{Iw1} and explicitly exhibited in~\cite{Iw2} that there is a section $\sigma:Q\to G_2Q$ such that 
$H_2^\sigma(\psi)$  is homotopic to the composition $\beta\circ\Sigma^2H_1(\alpha)$ for some $\beta$.
The section $\sigma$ extends to a section $\sigma':M_3\to G_3M_3$
(called a standard section in~\cite{Iw1}). Then we define $$\hat\sigma:X_2\to X_2'=(G_2Q\times S^2)\cup (S^2\times G_2Q)\cup (G_3M_3\times pt)\cup (pt\times G_3M_3)$$ to be the restriction of $\sigma'\times 1_{S^2}\cup 1_{S^2}\times \sigma'$ to $Y$. The space  $X_2'$ has a natural inclusion $\bar\xi:X'_2\to\bar X_2$. We define  
$$s=\bar i\circ\xi\circ\hat\sigma:X_2\to\hat G_3(M_3\times M_3).$$

For each attaching map $f\in\{\psi_-,\psi_+,\tilde\alpha\}$ we consider the  Fibration-Cofibration Lifting problem
$$
\begin{CD}
\hat  F_3(M_3\times M_3) @>i>> \hat G_3(M_3\times M_3) @>\hat p_3>> M_3\times M_3\\
@. @AsAA @ A\subset AA\\
S^{n}@>f>> X_2 @>j>\subset> C(f)\\
\end{CD}
$$
In Lemma~\ref{M} and Lemma~\ref{Q} we show that  each of the three lifting problems
has a solution. This defines a lift $\bar s:X_3\to \hat G_3(M_3\times M_3)$ of the inclusion $X_3\subset M_3\times M_3$ extending $s$.

Let $\bar s:X_3\to\hat G_3(M_3\times M_3)$ be such a lift. Then $\lambda\circ\bar s:X_3\to G_3(M_3\times M_3)$ is a lift of the inclusion $X_3\subset M_3\times M_3$ with respect to 
$p_3:G_3(M_3\times M_3)\to M_3\times M_3$ where $\lambda:\hat G_3(M_3\times M_3)\to G_3(M_3\times M_3)$ is a lift of $\hat p_3$ with respect $p_3$. Since the inclusion $X_3\to M_3\times M_3$ is a 29-equivalence, by Proposition~\ref{old} the inclusion of fibers  $F_3X_3\to F_3(M_3\times M_3)$ is a $(2\times3+29-1)$-equivalence.
In the pull-back diagram
$$
\begin{CD}
G_3X_3 @>q>> Z@>>> G_3(M_3\times M_3)\\
@. @Vp'VV @Vp_3VV\\
@. X_3 @>\subset >>M_3\times M_3.\\
\end{CD}
$$
the lift $\lambda\bar s$ defines a section $\nu':X_3\to Z$. The map $p^{X_3}$ factors as $p'\circ q$ where $q$ is a $34$-equivalence.
Since $\dim X_3=28$, there  is a lift of $\nu'$ to a section of $G_3X_3$. Therefore, $\cat X_3\le 3$.

\end{proof}
\begin{lemma}\label{M}
There are lifts $s_{1,2}: C(\psi_\pm)\to \hat G_3(M_3\times M_3)$ with respect to $\hat p_3:\hat G_3(M_3\times M_3)\to M_3\times M_3$ of the inclusion $C(\psi_\pm)\subset M_3\times M_3$ that extend $s$.
\end{lemma}
\begin{proof}
Since the inclusion $i:\hat F_3(M_3\times M_3)\to \hat G_3(M_3\times M_3)$ is injective on the homotopy groups, in view of Proposition~\ref{extract} it suffices to define
a null-homotopic map $\phi_\pm: S^{17}\to \hat F_3(M_3\times M_3)$ such that $i\circ\phi_\pm$ is homotopic to $s\circ\psi_\pm$.

The construction of $\phi_\pm$ is based on Iwase's proof of the inequality $\cat (M_3\times S^2)\le 3$~\cite{Iw1}. We consider the homotopy  commutative diagram from Proposition 3.7 in~\cite{Iw1} amended by the right column
$$
\begin{CD}
F_2M_3\ast F_0S^2 @>>> G_3M_3\times pt\cup G_2M_3\times G_1S^2 @>\subset>>\bar X_2\\
@A\bar jAA @A\subset AA @A\bar\xi AA\\
F_2Q\ast S^1 @. G_3M_3\times pt\cup G_2Q\times S^2  @>\subset>> X_2'\\
@A{H_2^{\sigma}(\psi)\ast 1_{S^1}}AA @AA{\sigma_1}A @A\hat\sigma AA\\
S^{15}\ast S^1 @>\bar\psi>> M_3\times pt\cup Q\times S^2 @>i_\pm>\subset> X_2\\
\end{CD}
$$
where  $\bar j$ is generated by the inclusions $Q\to M_3$ and $S^1\to\Omega\Sigma S^1=F_0S^2$, $\sigma_1$ is the restriction of $\hat\sigma$,
and $\psi$ is taken from Proposition~\ref{M3}.

 By Proposition~\ref{toda-p}, $\Sigma^4\psi=0$. 
Therefore, 
$$H_2^{\sigma}(\psi)\ast 1_{S^1}=\Sigma^2H_2^\sigma(\psi)=\Sigma^2\beta\circ\Sigma^4H_1(\alpha)=0.$$ 
Therefore, $\theta=\bar j\circ(H_2^{\sigma}(\psi)\ast 1_{S^1})$ is null-homotopic. Thus, we have a homotopy commutative diagram
$$
\begin{CD}
\hat F_3(M_3\times M_3) @>i>> \hat G_3(M_3\times M_3)\\
@A\phi_\pm AA  @AsAA\\
S^{17} @>\psi_\pm>> X_2
\end{CD}
$$
with null-homotopic $\phi_\pm=j_\pm\circ\eta_{2,0}\circ\theta$. Here $\eta_{2,0}:F_2M_3\ast F_0S^2\to\hat F_3(M_3\times S^2)$ and $j_\pm:\hat  F_3(M_3\times S^2)\to \hat  F_3(M_3\times M_3)$ are the inclusions induced by $i_\pm$. By Proposition~\ref{extract} there are the required lifts $s_1$ and $s_2$.
\end{proof}
\begin{lemma}\label{Q}
There is lifts $s_3: C(\tilde\alpha)\to \hat G_3(M_3\times M_3)$ with respect to $\hat p_3:\hat G_3(M_3\times M_3)\to M_3\times M_3$ of the inclusion $C(\tilde\alpha)\subset M_3\times M_3$ that extends $s$.
\end{lemma}
\begin{proof}
The construction of $s_3$ is based on Harper's proof~\cite{H} of the inequality $\cat(Q\times Q)\le 3$. 
We consider  the homotopy commutative diagram from~\cite{H} completed by the right commutative square
$$
\begin{CD}
F_1Q\ast F_1Q @>>> G_2Q\times G_1Q\cup G_1Q\times G_2Q@>\subset>> \bar X_2 \\
@AAH_1^\sigma(\alpha)\ast H_1^\sigma(\alpha)A @AA s' A @ AsAA\\
S^{13}\ast S^{13} @>\bar\alpha>> Q\times S^2\cup S^2\times Q @>\subset>> X_2\\
\end{CD}
$$
where $s'$ is the restriction of $s$.
 By the Barratt-Hilton formula (\cite{T}, Proposition 3.1) we obtain $$H_1^\sigma(\alpha)\ast H_1^\sigma(\alpha)=\Sigma (H_1(\alpha)\wedge H_1(\alpha))=
\Sigma(\Sigma^2H_1(\alpha)\circ \Sigma^{13}H_1(\alpha)).$$ Since $\Sigma^4H_1(\alpha)=0$, we obtain that the map $\theta=H_1^\sigma(\alpha)\ast H_1^\sigma(\alpha)$ is null-homotopic.

Again, we apply Proposition~\ref{extract} to the homotopy commutative diagram
$$
\begin{CD}
\hat F_3(M_3\times M_3) @>>> \hat G_3(M_3\times M_3)\\
@A\phi AA  @AsAA\\
S^{28} @>\tilde\alpha>> X_2
\end{CD}
$$
with null-homotopic $\phi=\hat j\circ\eta_{1,1}\circ\theta$ to obtain a lift $s_3$. Here $\eta_{1,1}:F_1Q\ast F_1Q\to\hat F_3(Q\times Q)$ and $\hat j:\hat  F_3(Q\times Q)\to \hat  F_3(M\times M_3)$ is the inclusion induced by the inclusion $Q\times Q\subset M_3\times M_3$.
\end{proof}

\begin{prop}
$\cat(M_2\times M_3)\ge 4$.
\end{prop}
\begin{proof}
We recall that both $M_2$ and $M_3$ have CW complex structures with one cell in each of the  dimensions 0,2,14, and 16. Consider the product CW-complex structure on $X=M_2\times M_3$. We show that $\cat X^{(18)}=4$. Then by Theorem~\ref{Stanley} $\cat X\ge 4$.
 Assume the contrary: $\cat X^{(18)}\le 3$. Consider the pull-back diagram generated by the 3-rd Ganea fibrations and the inclusion $j:M_2\times S^2\to X^{(18)}$
$$
\begin{CD}
G_3(M_2\times S^2) @>\xi>> Z @>>> G_3(X^{(18)})\\
@. @Vp'VV @Vp_3VV\\
@. M_2\times S^2 @>j>>  X^{(18)}\\
\end{CD}
$$
where $p'\circ\xi=p_3^{M_2\times S^2}$. By the assumption there is a section of $p_3$ which induces a section of $p'$.
Since the inclusion $j:M_2\times S^2\to X^{(18)}$ is a 13-equivalence and the spaces are 1-connected, by Proposition~\ref{old}  the mapping between fibers of $F_3(M_2\times S^2)\to F_3X^{(18)}$
is a $(2\times 3+13-1)$-equivalence. Hence, the homotopy fiber of $\xi$ is 18-connected. Since $\dim(M_2\times S^2)=18$, the section of $p'$ can be lifted to a section of $p^{M_2\times S^2}_3$.
This implies a contradiction: $\cat(M_2\times S^2)\le 3$. 
\end{proof}
\begin{rem} By Theorem~\ref{Stanley} the proper filtration $X^{(18)}\subset X^{(28)}\subset X^{(30)}\subset X$ defines a chain of inequalities
$\cat X^{(18)}\le\cat X^{(28)}\le\cat X^{(30)}\le \cat X$. In view of Theorem~\ref{BH}, whether any of the above inequalities is strict can be determined by Berstein-Hilton invariants. 
If one of  this three opportunities  for $\cat$ to jump up  is realized, then, as we show in the next section,
there is a counter-example to the Rudyak conjecture.
\end{rem}

\section{Rudyak Conjecture}

For a closed oriented manifold $M$ and $k\in\N$ by $kM$ we denote the connected sum of $k$ copies off $M$ and by $-kM$ we denote the connected sum $|k|\bar M$
where $\bar M$ is $M$ taken with the opposite orientation. The following is obvious
\begin{lem}\label{cof}
If $M_1,\dots, M_r$ are connected $n$-manifold, then there is a cofibration sequence
$$
\begin{CD}
\bigvee_{i=1}^{r-1}S^{n-1} @>>> \#_{j=1}^rM_j @>\psi>> \bigvee_{j=1}^rM_j.\\
\end{CD}
$$
\end{lem}

A special case of the following proposition was proven in~\cite{Dr}.

\begin{prop}\label{one}
Suppose that $g:N_1\to M_1$ and $h:N_2\to M_2$ are maps between closed manifolds of degree $p$ and $q$ for mutually prime $p$ and $q$. 
Then there are  $k,\ell\in \Z$ and a degree one map
$$
f:k(M_1\times N_2)\#\ell(N_1\times M_2)\to M_1\times M_2.
$$
\end{prop}
\begin{proof}
Let $\dim N_1=n_1$ and $\dim N_2=n_2$. Take $k$ and $\ell$ such that $\ell p+kq=1$.
We may assume that the above connected sum is obtained by taking the wedge of $(|k|+|\ell|-1)$ copies of $(n_1+n_2-1)$-spheres embedded in one of the summands and gluing all other summands along those spheres.
Consider the cofibration map from Lemma~\ref{cof}
$$
\psi:k(M_1\times N_2)\# \ell(N_1\times M_2) 
\to\bigvee_k(M_1\times N_2)\vee \bigvee_\ell (N_1\times M_2).
$$
Let the map $$\phi:\bigvee_k(M_1\times N_2)\vee \bigvee_\ell (N_1\times M_2) \to M_1\times M_2$$
be defined as the union 
$$
\phi=\bigcup_k(1\times g)\cup\bigcup_\ell(h\times 1).
$$
The degree of $h\times 1$ is $p$ and  the degree of $1\times g$ is $q$. Hence the degree of
$f=\phi\circ\psi$ is $\ell p+kq=1$.
\end{proof}

\begin{thm}\label{rudyak}
Suppose that $g:N_1\to M_2$ and $h:N_2\to M_2$ are maps between closed manifolds of degree $p$ and $q$ for mutually prime $p$ and $q$ and 
$$\max\{\cat(M_1\times N_2),\cat(N_1\times M_2)\}<\cat(M_1\times M_2).$$
Then there is a counter-example to Rudyak's Conjecture.
\end{thm}
\begin{proof}
By Proposition~\ref{one} there is a degree one map
$$
f:k(M_1\times N_2)\#\ell(N_1\times M_2)\to M_1\times M_2.
$$
By the connected sum formula~($\ast$), $$\cat(k(M_1\times N_2)\#\ell(N_1\times M_2))\le\max\{\cat(M_1\times N_2,\cat(N_1\times M_2)\}.$$
\end{proof}

\begin{cor}
If $\cat(M_2\times M_3)\ge 5$, then there is a counter-example to Rudyak's conjecture.
\end{cor}
\begin{proof}
By Proposition~\ref{M2} and Proposition~\ref{M3} there are maps of degree two, $g:S^{14}\times S^2\to M_2$, and of degree three, $h:S^{14}\times S^2\to M_3$.
Then the map $$f:-(M_2\times S^{14}\times S^2)\#2(S^{14}\times S^2\times M_3)\to M_2\times M_3$$ of Proposition~\ref{one}
has degree one. We note that $$\cat(M_2\times S^{14}\times S^2)\le\cat(M_2\times S^{14})+\cat S^2=3+1=4$$ and
$\cat(S^{14}\times S^2\times M_3)\le\cat S^2+\cat(S^{14}\times M_3)=4$.
\end{proof}

\end{document}